\documentclass[english]{article}
\usepackage[T1]{fontenc}
\usepackage[latin9]{inputenc}
\usepackage{geometry}
\usepackage{babel}
\usepackage{float}
\usepackage{amsthm}
\usepackage{amsmath}
\usepackage{amssymb}
\usepackage{graphicx}
\usepackage[authoryear]{natbib}
\usepackage[unicode=true,pdfusetitle,
 bookmarks=true,bookmarksnumbered=false,bookmarksopen=false,
 breaklinks=false,pdfborder={0 0 1},backref=false,colorlinks=false]
 {hyperref}

\makeatletter

\newcommand{\noun}[1]{\textsc{#1}}
\floatstyle{ruled}
\newfloat{algorithm}{tbp}{loa}
\providecommand{\algorithmname}{Algorithm}
\floatname{algorithm}{\protect\algorithmname}

\theoremstyle{plain}
\newtheorem{thm}{\protect\theoremname}
  \theoremstyle{definition}
  \newtheorem{problem}[thm]{\protect\problemname}
  \theoremstyle{plain}
  \newtheorem{lem}[thm]{\protect\lemmaname}
  \theoremstyle{remark}
  \newtheorem{claim}[thm]{\protect\claimname}
  \theoremstyle{definition}
  \newtheorem{defn}[thm]{\protect\definitionname}
  \theoremstyle{plain}
  \newtheorem{prop}[thm]{\protect\propositionname}

\bibpunct{(}{)}{,}{a}{,}{,}

\makeatother

  \providecommand{\claimname}{Claim}
  \providecommand{\definitionname}{Definition}
  \providecommand{\lemmaname}{Lemma}
  \providecommand{\problemname}{Problem}
  \providecommand{\propositionname}{Proposition}
\providecommand{\theoremname}{Theorem}

\begin{document}

\title{A Hough Transform Approach to Solving Linear Min-Max Problems}

\author{Carmi Grushko\\
carmi.grushko@gmail.com\\
\\
Technion, Israel Institute of Technology}

\maketitle
Several ways to accelerate the solution of 2D/3D linear min-max problems
in $n$ constraints are discussed. We also present an algorithm for
solving such problems in the 2D case, which is superior to \noun{Cgal}'s
linear programming solver, both in performance and in stability.

\section{Purpose}

This work is focused on several ways to accelerate the solution of
2D/3D linear min-max problems in $n$ constraints. We also present
an algorithm for solving such problems in the 2D case, which is superior
to \noun{Cgal}'s linear programming solver, both in performance and
in stability.
\begin{problem}
Linear Min-Max problem, also known as $L_{\infty}$ linear optimization
\[
\arg\min_{x,y\in\mathbb{R}}\left(\max_{i\in\left[n\right]}\left|a_{i}x+b_{i}y+c_{i}\right|\right)=\arg\min_{x,y}\left(\left\Vert M\cdot\left(x,y,1\right)^{T}\right\Vert _{\infty}\right),
\]
where $M=\left[a\vert b\vert c\right]$.
\end{problem}
This problem can be re-written as a Linear Programming of the following
form.
\begin{problem}
\label{fig:Central-Min-Max-problem. (no last two constraints)}
\begin{eqnarray*}
\mbox{minimize}_{x,y,t} &  & t\\
\mbox{s.t} &  & a_{1}x+b_{1}y+c_{1}\leq t\\
 &  & \,\,\,\,\,\,\,\,\,\,\,\,\,\,\vdots\\
 &  & a_{n}x+b_{n}y+c_{n}\leq t.
\end{eqnarray*}

\end{problem}
It is known for several decades (\citet{Megiddo}) that general Linear
Programming problems (and thus, also linear min-max problems) can
be solved in $\mathcal{O}\left(n\right)$ time when the dimension
is constant. Therefore, there is only hope to demonstrate a constant
factor acceleration; however, in real applications, this can be valuable.

\section{The Hough Transform}

We begin with a quick overview of a natural extension to the Hough
Transform \citet{Hough} to 3D. We note that the following theorems,
although proved for 3D, hold equally well in 2D; the proofs are identical
if we rename the $z$ coordinate into $y$.

Given a point $p=\left(a,b,c\right)$ in $\mathbb{R}^{3}$, we define
its dual plane as $\mathcal{H}\left(p\right)=\ell\left(x,y\right)=ax+by-c$,
and given a plane $\pi\left(x,y\right)=ax+by+c$, we define its dual
point as $\mathcal{H}\left(\pi\right)=\left(a,b,-c\right)$. The usefulness
of these definitions is highlighted in the following lemmas.
\begin{lem}
\label{lem:p > pi iff H(p) > H(pi)}A point $p=\left(x_{0},y_{0},z_{0}\right)$
is above a plane $\pi\left(x,y\right)=ax+by+c$ iff the plane $\mathcal{H}\left(p\right)$
is below the point $\mathcal{H}\left(\pi\right)$. Moreover, $p$
is on $\pi$ iff $\mathcal{H}$$\left(\pi\right)$ is on $\mathcal{H}\left(p\right)$.\end{lem}
\begin{proof}
We know that $\pi\left(x_{0},y_{0}\right)=x_{0}a+y_{0}b+c<z_{0}$;
by definition of the Hough transform, $\left[\mathcal{H}\left(p\right)\right]\left(\mathcal{H}\left(\pi\right)_{x},\mathcal{H}\left(\pi\right)_{y}\right)=x_{0}\mathcal{H}\left(\pi\right)_{x}+y_{0}\mathcal{H}\left(\pi\right)_{y}-z_{0}$
which equals $x_{0}a+y_{0}b-z_{0}$. Because $x_{0}a+y_{0}b+c<z_{0}$,
we must have $x_{0}a+y_{0}b-z_{0}<-c$, concluding that $\left[\mathcal{H}\left(p\right)\right]\left(\mathcal{H}\left(\pi\right)_{x},\mathcal{H}\left(\pi\right)_{y}\right)<\mathcal{H}\left(\pi\right)_{z}$.
A very similar argument can be used to show that a point is on a plane
iff its dual is on its dual. \end{proof}
\begin{lem}
The upper envelope of a set of planes corresponds to the lower convex
hull of the planes' dual points.\end{lem}
\begin{proof}
By definition, a point $p$ on the upper envelope of a set of planes
$\left\{ \pi_{i}\right\} $ is above all of them (except several on
which it lies), and by Lemma \ref{lem:p > pi iff H(p) > H(pi)} we
have that the plane $\mathcal{H}\left(p\right)$ is below all the
points $\mathcal{H}\left(\pi_{i}\right)$ (except several which it
touches). The converse is also true - every plane $\pi$ that is part
of the lower convex hull is lower than all points $p_{i}$ (except
several which it touches), therefore $\mathcal{H}\left(\pi\right)$
must be above all the planes $\mathcal{H}\left(p_{i}\right)$.
\end{proof}
We will henceforth denote the convex hull of a set of points $P$
by $\mathcal{CH}\left(P\right)$, and its lower convex hull by $\mathcal{LH}\left(P\right)$.
Figure \ref{fig:Hough-Transform-of} demonstrates the above lemmas
in the 2D case. 
\begin{figure}
\begin{centering}
\includegraphics[scale=0.65]{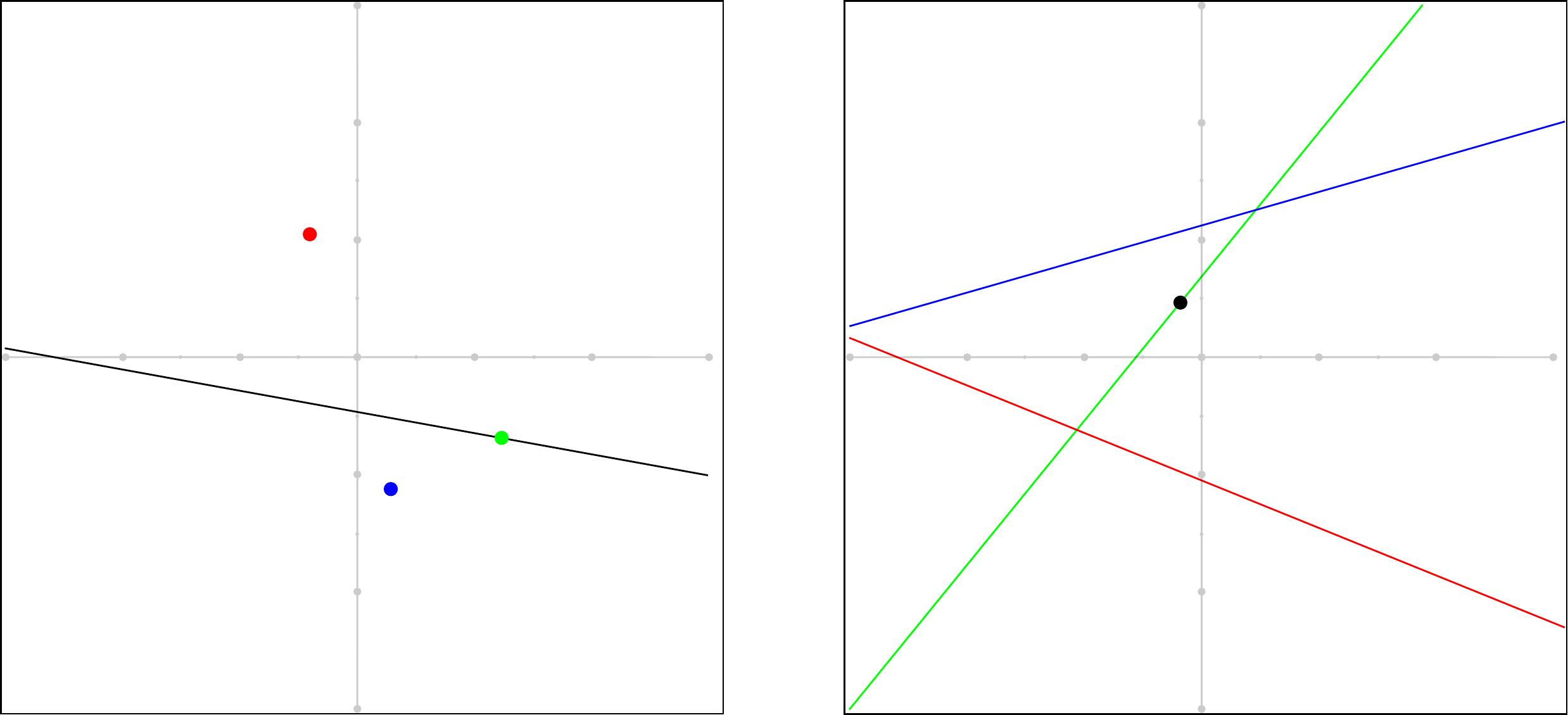}
\par\end{centering}

\caption{\label{fig:Hough-Transform-of}Hough Transform of three points and
a line (plane in 3D). Lines and points, and their duals, are related
by their color.}

\end{figure}

\section{A Linear Min-Max Problem in the Hough space}

Solving the linear program is equivalent to finding the lowest point
of the upper envelope of the set of planes defined by 
\[
z_{k}\left(x,y\right)=a_{k}x+b_{k}y+c\,\,\,\, k=1,2,\dots,n.
\]

We saw that the upper envelope corresponds to the lower convex hull
of the set of points $DP=\left\{ \left(a_{k},b_{k},-c_{k}\right)\right\} _{k=1}^{n}$.
It is now obvious that a solution to Problem \ref{fig:Central-Min-Max-problem. (no last two constraints)},
which is a point $\left(x,y,t\right)$ on the upper envelope, corresponds
to a plane in the dual space, which is defined by a face of $\mathcal{LH}\left(DP\right)$. 

Consider an optimum for the target function of the linear problem,
namely $t_{opt}$, and define the following plane which is parallel
to the $x-y$ plane: $\pi\left(x,y\right)=0\cdot x+0\cdot y+1\cdot t_{opt}$.
The dual to this plane, $\mathcal{H}\left(\pi\right)$, is a point
on the $z$ axis. Because the optimal solution $p$ to Problem \ref{fig:Central-Min-Max-problem. (no last two constraints)}
is a point that is on $\pi$, its dual $\mathcal{H}\left(p\right)$
is a plane on which the point $\mathcal{H}\left(\pi\right)$ must
be. This means that the dual to the plane that has the highest intersection
with the $z=0$ axis (and is part of $\mathcal{LH}\left(DP\right)$)
is the solution to the linear program. Moreover, it is obvious that
if there is a face of the lower convex hull that intersects the $z$
axis, it has the highest intersecting plane. If there is no such face,
either all of the points have a positive $a_{k}$, or all have a negative
$a_{k}$, which means the problem is unbounded.

The last two statements suggest that solving our linear program is
equivalent to finding the face of the lower convex hull that intersects
the $z$ axis. Unfortunately, we are unaware of an $\mathcal{O}\left(n\right)$
method to do so.

Figure \ref{fig:Dual of linear min-max} demonstrates the above lemmas
in the 2D case. 
\begin{figure}
\begin{centering}
\includegraphics[scale=0.65]{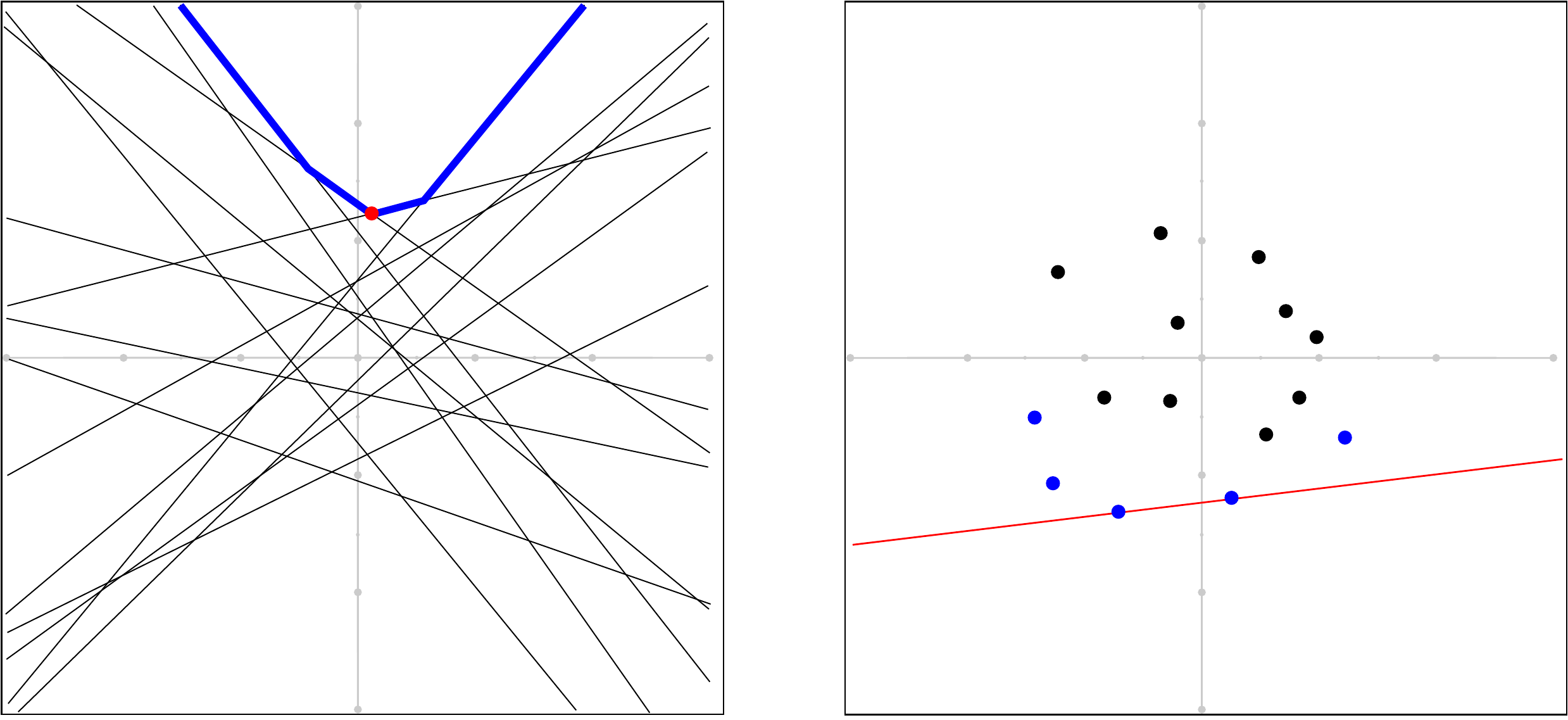}
\par\end{centering}

\caption{\label{fig:Dual of linear min-max}Left: a linear min-max problem,
and its solution (red point). Right: duals of planes defining the
problem, and the dual to the solution (red line).}
\end{figure}

\section{The 2D Case: Solving Problem \ref{fig:Central-Min-Max-problem. (no last two constraints)}
in the Hough plane}

Although the last section ends in a pessimistic note, this section
provides an algorithm which in practice, solves 2D linear min-max
problems much faster than \noun{Cgal}'s \noun{solve\_linear\_program}
\citet{cgal:fgsw-lqps-11}. We were unable, at the time of writing
this work, to provide a complexity proof; however, experiments support
the conjecture it is linear in the number of constraints. See Figure
\ref{fig:CPU-time-in} 
\begin{figure}
\begin{centering}
\includegraphics[scale=0.65]{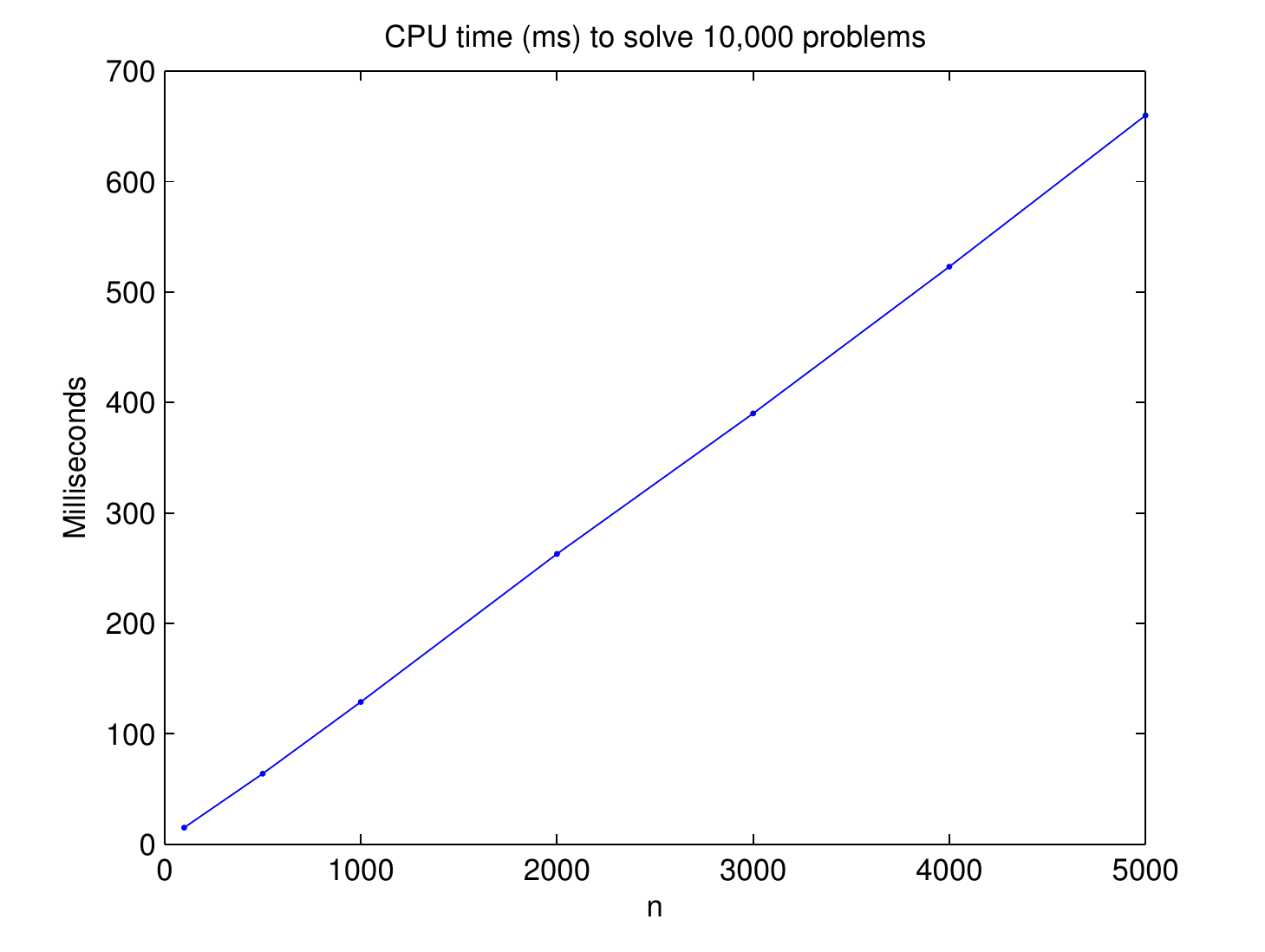}
\par\end{centering}

\caption{\label{fig:CPU-time-in}CPU time in milliseconds to solve 10,000 random
problems with $n$ constraints (our algorithm)}
\end{figure}

To clarify: we remove the $y-$coordinate from our problem formulation
and rename $z$ into $y$, to obtain the following class of problems:
\begin{eqnarray*}
\mbox{minimize}_{x,t} &  & t\\
\mbox{s.t} &  & a_{1}x+b_{1}\leq t\\
 &  & \,\,\,\,\,\,\,\,\,\,\,\,\,\,\vdots\\
 &  & a_{n}x+b_{n}\leq t.
\end{eqnarray*}

The algorithm is as follows.

\begin{algorithm}[H]
Input: A set of $n$ constraints: $a_{i}x+b_{i}\leq t,\,\,\,\, i\in\left[n\right]$

Output: Optimal primal point $\left(x,t\right)$
\begin{enumerate}
\item Transform constraints into a set of points $DP$
\item Partition $DP$ into $L$ (points with negative $x-$coordinate) and
$R$ (positive)
\item Pick some point $p_{0}\in L$
\item Repeat until no change:

\begin{enumerate}
\item Given $p_{i}$ in $L$ ($R$), find a point $p_{i+1}$ in $R$ ($L$)
with the largest clockwise (counter-clockwise) turn from $p_{i}$.
\end{enumerate}
\item Compute the line that intersects the last two points $p_{k}$ and
$p_{k+1}$: $\ell\left(x\right)=mx-mp_{k}^{\left(x\right)}+p_{k}^{\left(y\right)}$
where $m=\left(p_{k+1}^{\left(y\right)}-p_{k}^{\left(y\right)}\right)/\left(p_{k+1}^{\left(x\right)}-p_{k}^{\left(x\right)}\right)$
\item Compute the dual to this plane: $\left(m,mp_{k}^{\left(x\right)}-p_{k}^{\left(y\right)}\right)$
\item Return either (6) or $-\infty$.
\end{enumerate}
\caption{\label{alg:Solving-in-Hough-2D}Solving in Hough plane (2D)}
\end{algorithm}

First, we should prove the algorithm takes a finite number of steps.
\begin{claim}
Step (4) in Algorithm \ref{alg:Solving-in-Hough-2D} terminates.\end{claim}
\begin{proof}
Suppose we are at step $i+1$, and the last two points are $p_{i}$
and $p_{i+1}$. Also assume without loss of generality that $p_{i}\in L$
and therefore $p_{i+1}\in R$. Next, either step (4a) terminates (we're
done) or a new point $p_{i+2}\in L$ is selected. The criterion for
the selection is that $p_{i+2}$ constitutes a counter-clockwise turn
from $p_{i}$, which is equivalent to to the line $\left(p_{i+2},p_{i+1}\right)$
having a larger inclination than the line $\left(p_{i},p_{i+1}\right)$.
However, because both lines intersect $p_{i+1}$, and the second one
has a larger inclination, its intersection with the $y$ axis is smaller
than that of the first one. This means that the $y-$intersection
of any $\left(p_{i},p_{i+1}\right)$ is smaller than the $y-$intersection
of any line $\left(p_{j},p_{j+1}\right)$ for all $i>j$. Because
the number of points is finite, the minimum of $y-$intersections
of lines defined by any pair of points in $DP$ is finite, therefore
step (4a) terminates.
\end{proof}
We are left with showing that when step (4a) terminates, the points
$p_{k}$ and $p_{k+1}$ are indeed part of $\mathcal{LH}\left(DP\right)$;
the fact that the segment $\left(p_{k},p_{k+1}\right)$ intersects
the $y$ axis is obvious.
\begin{proof}
Assume without loss of generality that $p_{k}\in L$ and therefore
$p_{k+1}\in R$. By step (4a) we know that no point in $R$ has a
clockwise turn from the line $\left(p_{k},p_{k+1}\right)$, or equivalently,
that all the points in $R$ are above this line. Similarly, no point
in $L$ is above this same line for symmetric reasons. This means
this line supports $DP$ (from below), which means it must be on its
lower convex hull. 
\end{proof}
Figure \ref{fig:2D algo example} demonstrates the algorithm. 
\begin{figure}
\begin{centering}
\includegraphics[scale=0.65]{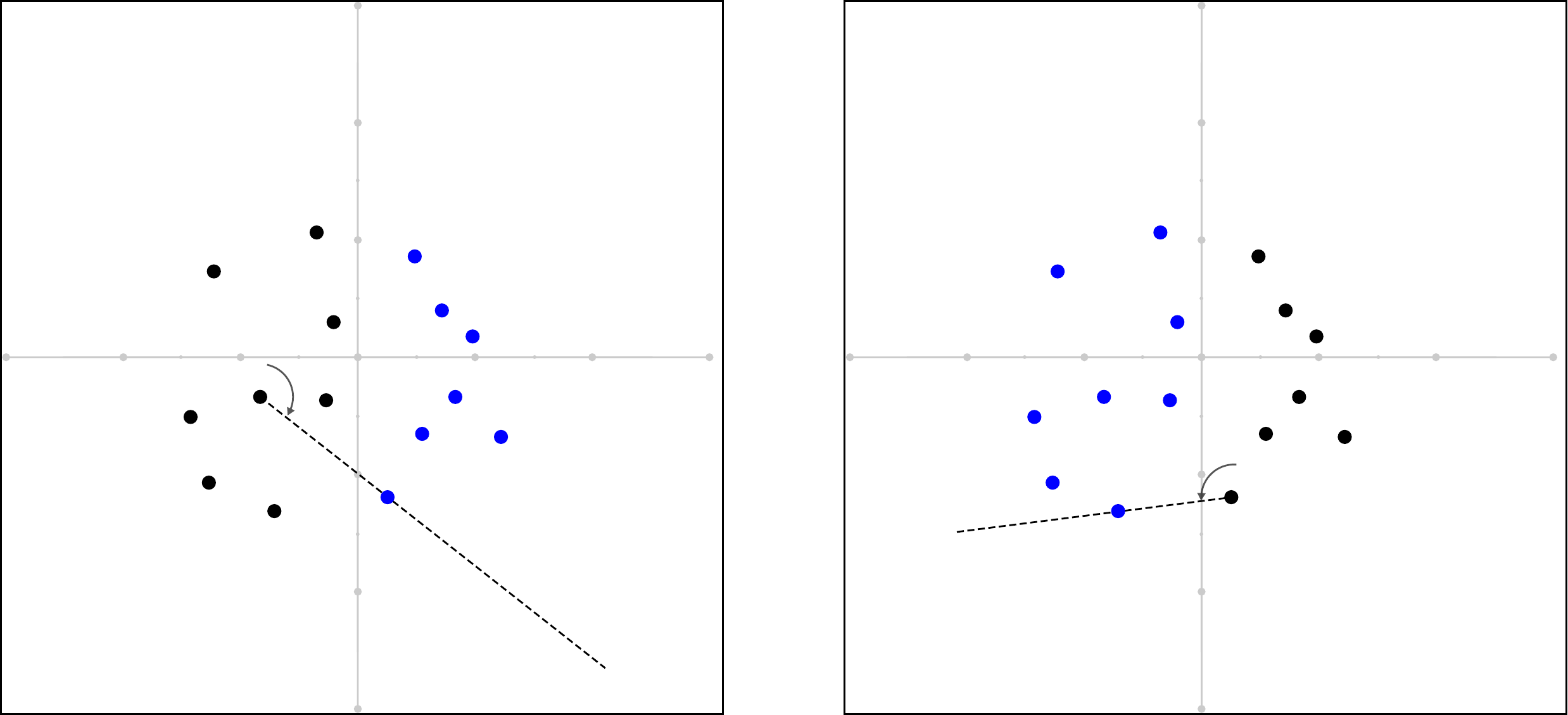}
\par\end{centering}

\caption{\label{fig:2D algo example}Two applications of step 4a}
\end{figure}

\subsection{Experimental Verification}

We pit our algorithm against \noun{Cgal}'s linear programming solver
\citet{cgal:fgsw-lqps-11}. The constraints were drawn from a 2D Gaussian
random variable $X\sim\left[N\left(E=0,\sigma^{2}=10\right)\right]^{2}$.
As is evident in Figure \ref{fig:Ratio-of-time} which depicts the
ratio between running times of the two solvers, versus the number
of constraints $n$, the new algorithm is about 10 times faster than
\noun{Cgal}'s.

\begin{figure}
\centering{}\includegraphics[scale=0.65]{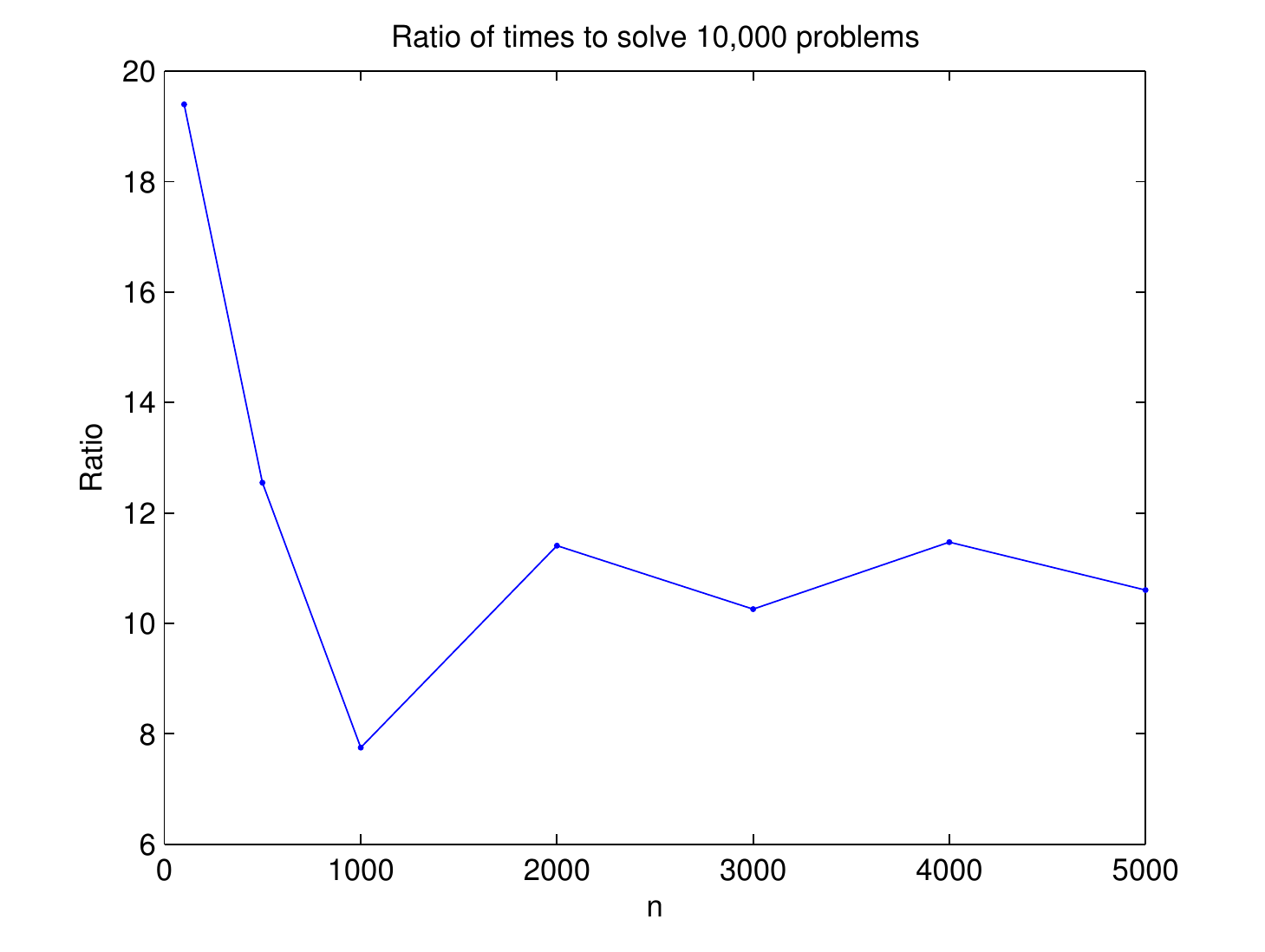}\caption{\label{fig:Ratio-of-time}Ratio of time to solve 10,000 problems,
vs. problem size ($n$)}
\end{figure}

\subsubsection{Implementation notes}

Our solver only uses addition and multiplication and is highly-parallelizable.
No divisions are used, which results in a faster and more stable solver. 

\noun{Cgal}'s solver requires the use of so-cold exact types, such
as rational numbers or arbitrary-precision floating point numbers;
the use of such types is extremely slow at this time, because of the
way \noun{Cgal} uses them. Therefore, the solver was tricked to using
simple machine double-precision floating point numbers. This resulted
in numerical failure when the number of constraints $n$ was greater
than $\approx10^{5}$. Our solver, however, is working even with this
number of constraints. Moreover, the only numerically sensitive step
in our algorithm is in Step (4a), which involves deciding if three
points define a clockwise turn. This is done by computing $\left(p_{1}-p_{0}\right)\wedge\left(p_{2}-p_{0}\right)\overset{?}{>}0$,
which can be done by setting $\tilde{p}_{i}=p_{i}-p_{0}$ and then
deciding if 
\[
\tilde{p}_{1}^{x}\cdot\tilde{p}_{2}^{y}\mbox{ is greater than }\tilde{p}_{1}^{y}\cdot\tilde{p}_{2}^{x},
\]
or equivalently,
\[
\frac{\tilde{p}_{1}^{x}}{\tilde{p}_{1}^{y}}\mbox{ is greater than }\frac{\tilde{p}_{2}^{x}}{\tilde{p}_{2}^{y}}.
\]
This comparison can be performed in exact by converting each fraction
to a continued  fraction, which is still much faster than using an
arbitrary-precision floating point number.

A computer using Core 2 Quad (Q9400) @ 2.66 GHz and 6GB RAM was used
for the benchmarking. 10,000 random problems were solved by each of
the solvers, and their results compared.

\section{The 3D Case: Discarding of Constraints}

To again recede to a pessimistic note, we were unable to extend the
2D algorithm into 3D. However, modifying the problem at hand allows
us to quickly and safely discard of some of the constraints. The modified
problem has two more constraints: the solution point $\left(x,y,t\right)$
must have its coordinates $x$ and $y$ between 0 and 1.
\begin{problem}
\label{fig:Central-Min-Max-problem.}
\begin{eqnarray*}
\mbox{minimize}_{x,y,t} &  & t\\
\mbox{s.t} &  & a_{1}x+b_{1}y+c_{1}\leq t\\
 &  & \,\,\,\,\,\,\,\,\,\,\,\,\,\,\vdots\\
 &  & a_{n}x+b_{n}y+c_{n}\leq t\\
 &  & x\in\left[0,1\right]\\
 &  & y\in\left[0,1\right].
\end{eqnarray*}

\end{problem}
This problem is indeed very specific, but is in the core of the GMDS
algorithm \citet{GMDS} in the $L_{\infty}$ norm, where $x$ and
$y$ represent barycentric coordinates inside a triangle. 

To reiterate, solving Problem \ref{fig:Central-Min-Max-problem.}
is equivalent to finding the face of the 3D lower convex hull of the
points which are dual to the planes defined by the constraints, which
intersects the $z$ axis. In addition, if this (only) plane $z\left(x,y\right)=ax+by+z$
has either $a\notin\left[0,1\right]$ and/or $b\notin\left[0,1\right]$,
the solution must be on the boundary of the feasible set. 

In other words, it is possible to discard of all of the points in
$DP$ which support planes with either $a\notin\left[0,1\right]$
and/or $b\notin\left[0,1\right]$ and not change the solution, provided
that the boundaries of the feasible set are checked. We note that
$a$ corresponds to the inclination in the $x$ direction, and $b$
to the inclination in the $y$ direction.

To illustrate the points which can be discarded, consider a simplification
of the problem to 2D. Figure \ref{fig:Points-in-red} shows a set
of points for which we should find a line that supports the rest of
the points, from below, and has the highest $y-$intersect. Surely,
we can discard of the points on the left of the lowest point, because
these points are either interior, or support segments of $\mathcal{LH}\left(DP\right)$
which have negative inclinations. Also, points on the right of the
lowest point, which define lines with inclination greater than 1 can
also be discarded. 
\begin{figure}
\begin{centering}
\includegraphics[scale=0.7]{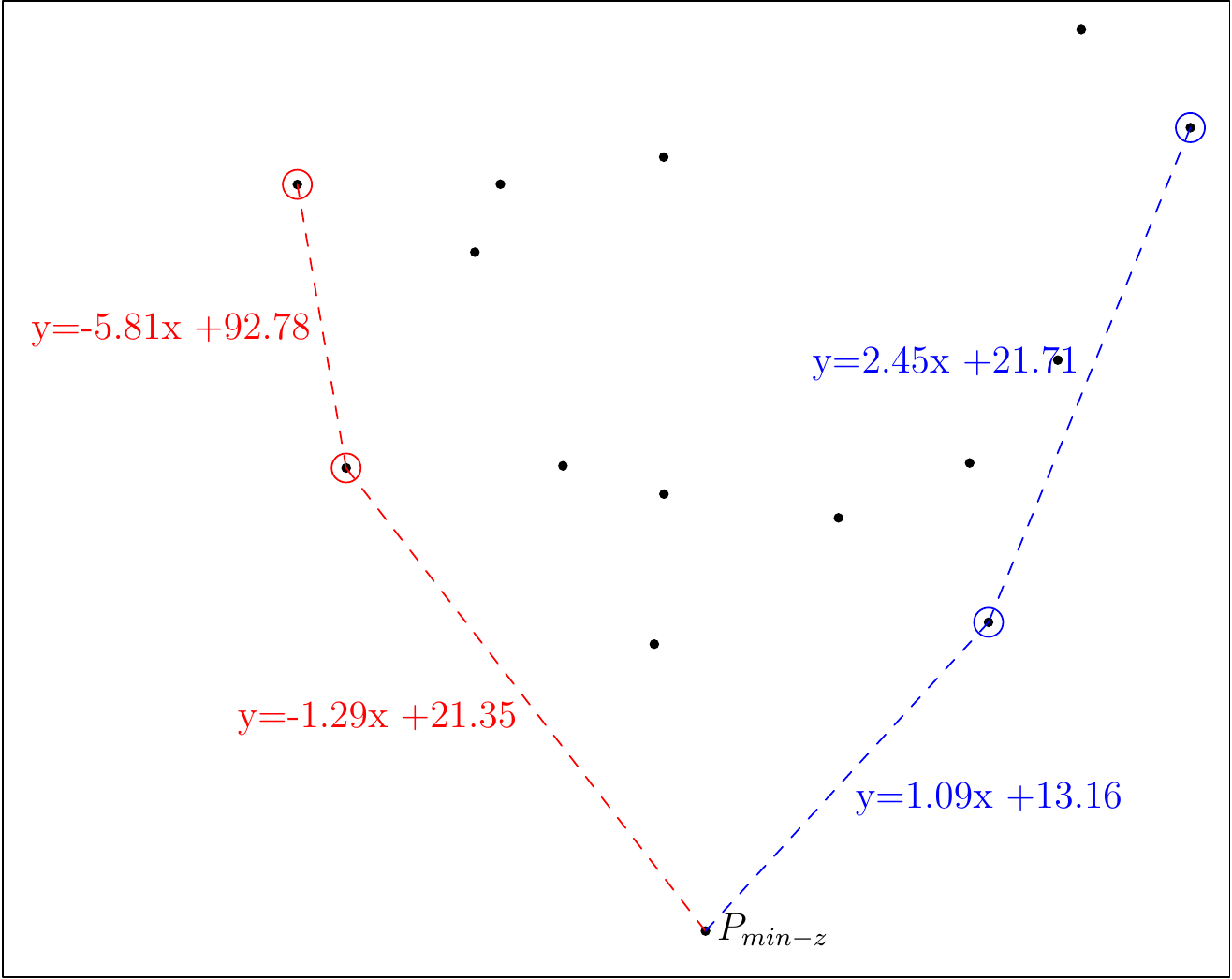}
\par\end{centering}

\caption{\label{fig:Points-in-red}Points in red only support segments which
have a negative inclination. Blue ones only support segments with
inclinations greater than 1.}
\end{figure}

The last two statements are true in 2D, but not necessarily in 3D
- removing points changes the convex hull, and there is danger that
the new convex hull will introduce bogus solutions to the problem.
We will see later that this is not the case, and removing the points
is indeed safe.

The main result is detailed in Theorem \ref{thm:Discarding-of-points is ok},
which is based on the following propositions.
\begin{defn}
A point $p=\left(p_{x},p_{y},p_{z}\right)$ is \emph{behind} another
point $q=\left(q_{x},q_{y},q_{z}\right)$ if $p_{x}<q_{x}$, $p_{y}<q_{y}$
and $p_{z}>q_{z}$.\end{defn}
\begin{prop}
\label{prop:behind cant be a solution}Let $p^{min}\in DP$ be a point
with a minimal $z-$coordinate, and let $p\in DP$ be point which
is behind\emph{ $p^{min}$.} Then any plane defined by $n\cdot\left(q-p\right)=0$
that:
\begin{enumerate}
\item Goes through $p$,
\item Supports $\mathcal{CH}\left(DP\right)$, that is, for any point $q\in DP$
we have $n\cdot\left(q-p\right)\geq0$, and;
\item Has a positive $z-$coordinate $n_{z}$
\end{enumerate}
cannot be a solution to Problem \ref{fig:Central-Min-Max-problem.}
, because its dual point has negative $x$ and/or $y$ coordinate.\end{prop}
\begin{proof}
Apply Requirement (2) to the point $q=p^{min}$, for which $q-p$
equals $\left(+\alpha,+\beta,-\gamma\right)$ for some positive $\alpha,\beta$
and $\gamma$. The result is the relation $n_{x}\alpha+n_{y}\beta>n_{z}\gamma$,
which forces either of $n_{x}$ or $n_{y}$ to be positive. Now, because
$\left(n_{x},n_{y},n_{z}\right)$ is a normal to a plane, the corresponding
plane equation must be $n_{z}\cdot z\left(x,y\right)=\left(-n_{x}\right)x+\left(-n_{y}\right)y+c$
for some constant $c$. Because $n_{z}$ is positive by Requirement
(3), either of the coefficients of the plane must be negative, which
means the dual to this plane cannot be a solution to Problem \ref{fig:Central-Min-Max-problem.}.
\end{proof}
We can state a similar result for points which are ``too steep''
with respect to $p^{min}$.
\begin{defn}
A point $p=\left(p_{x},p_{y},p_{z}\right)$ is \emph{too steep with
respect to} another point $q=\left(q_{x},q_{y},q_{z}\right)$ if $p\succ q$
(larger in all coordinates), and 
\[
\frac{p_{z}-q_{z}}{p_{x}-q_{x}}>1\mbox{ and }\frac{p_{z}-q_{z}}{p_{y}-q_{y}}>1.
\]
\end{defn}
\begin{prop}
\label{prop:too steep cant be solution}Let $p^{min}\in DP$ be a
point with a minimal $z-$coordinate, and let $p\in DP$ be a point
which is \emph{too steep with respect to} $p^{min}$. That is, with
all coordinates larger than those of $p^{min}$ such that the vector
$p^{min}-p=\left(-\alpha,-\beta,-\gamma\right)$, $\alpha,\beta$
and $\gamma$ positive, satisfies 
\[
\alpha<\gamma\mbox{ and }\beta<\gamma.
\]
Then any plane defined by $n\cdot\left(q-p\right)=0$ that:
\begin{enumerate}
\item Goes through $p$,
\item Supports $\mathcal{CH}\left(DP\right)$, that is, for any point $q\in DP$
we have $n\cdot\left(q-p\right)\geq0$, and;
\item Has a positive $z-$coordinate $n_{z}$
\end{enumerate}
cannot be a solution to Problem \ref{fig:Central-Min-Max-problem.},
because its dual point has its $x$ and/or $y$ coordinate larger
than 1.\end{prop}
\begin{proof}
In a similar way to the proof of Proposition \ref{prop:behind cant be a solution},
we obtain 
\[
\left(n_{x},n_{y},n_{z}\right)\cdot\left(-\alpha,-\beta,-\gamma\right)=\left(-n_{x}\right)\alpha+\left(-n_{y}\right)\beta-n_{z}\gamma,
\]
which implies
\[
\left(-\frac{n_{x}}{n_{z}}\right)\alpha+\left(-\frac{n_{y}}{n_{z}}\right)\beta>\gamma.
\]
We can normalize $n$ such that $n_{z}$ will be 1, without changing
its sign, because $n_{z}$ is positive by Requirement (3):
\[
\left(-n_{x}\right)\frac{\alpha}{\gamma}+\left(-n_{y}\right)\frac{\beta}{\gamma}>1.
\]
Now, if both $\alpha/\gamma$ and $\beta/\gamma$ are smaller than
1, $n_{x}$ and $n_{y}$ cannot be both larger than -1. Noting that
the explicit plane equation which the normal vector $n$ (with $n_{z}=1$)
induces is $z\left(x,y\right)=\left(-n_{x}\right)x+\left(-n_{y}\right)y+c$
for some $c$, leads to the conclusion that at least one of the coefficients
of the plane must be larger than 1, and therefore its dual point cannot
be a solution to the linear program.
\end{proof}
We state a simple characterization of lower convex hulls :
\begin{lem}
\label{lem:nz positive}Any plane $\pi$ defined by $n\cdot\left(q-p\right)=0$,
which supports $\mathcal{LH}\left(DP\right)$ must have a positive
$z-$coordinate of $n$ , $n_{z}$.\end{lem}
\begin{proof}
Since $\pi$ is part of the lower convex hull, there must be a pair
of points $q_{1}$ and $q_{2}$ such that $q_{1}$ is on $\pi$, and
$q_{2}=q_{1}+\epsilon\left(0,0,1\right)$; therefore, the condition
$n\cdot\left(q_{1}-q_{2}\right)>0$ forces $n_{z}>0$. 
\end{proof}
We are now ready to show that we can safely discard of a class of
constraints.
\begin{thm}
\label{thm:Discarding-of-points is ok}Discarding of points which
are behind $p^{min}$ (as defined in Proposition \ref{prop:behind cant be a solution}),
or are too steep (as defined in Proposition \ref{prop:too steep cant be solution})
leads to a min-max problem which is equivalent to Problem \ref{fig:Central-Min-Max-problem.}.\end{thm}
\begin{proof}
We will prove for points which are behind $p^{min}$; the proof for
points which are too steep is almost identical, and is left for the
reader.

Let $p$ be a point which is behind $p^{min}$. First, any face which
$p$ supports and is part of $\mathcal{LH}\left(DP\right)$ cannot
be a solution, therefore removing $p$ does not change the solution
to Problem \ref{fig:Central-Min-Max-problem.} (apply Lemma \ref{lem:nz positive}
and Proposition \ref{prop:behind cant be a solution}). However, there
still remains a possibility that removing $p$ creates new faces in
$\mathcal{LH}\left(DP\right)$ which results in a bogus solution to
Problem \ref{fig:Central-Min-Max-problem.}. We now show it is not
the case.

We consider the incremental convex hull construction algorithm detailed
in \citet{berg2000computational}. Suppose we have $A=\mathcal{LH}\left(DP\backslash p\right)$,
and would like to construct $B=\mathcal{LH}\left(DP\right)$. This
is done by finding all the faces of $A$ which are \emph{visible}
to $p$, that is, all faces which separate $p$ from $DP\backslash p$.
These faces are then removed, and the hole is filled using faces that
$p$ supports. 

A convex hull of a set of points $P$ is unique; one way to see this
is to recall one definition of $\mathcal{CH}\left(P\right)$ - the
intersection of all convex sets that contain $P$. This uniqueness
implies that the faces that are added to the convex hull, when we
remove $p$, are the faces that would have been removed if we constructed
$B$ from $A$ using the aforementioned algorithm. This characterizes
the faces that are added to $A$ when we remove $p$ - they all define
planes which separate $p$ from $DP\backslash p$. See Figure \ref{fig:Incremental-convex-hull}.
\begin{figure}
\begin{centering}
\includegraphics[scale=0.65]{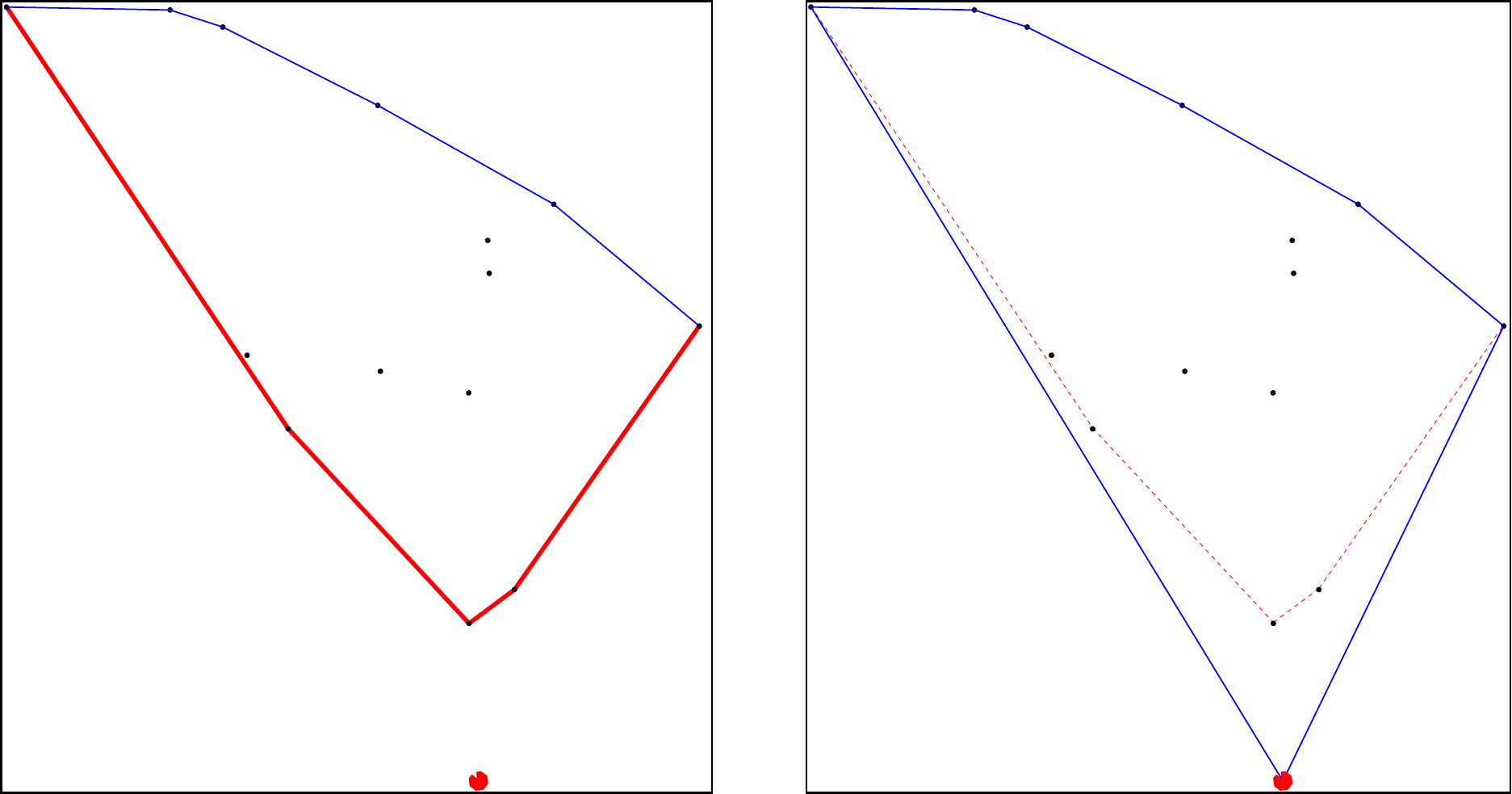}
\par\end{centering}

\caption{\label{fig:Incremental-convex-hull}Incremental convex hull construction.
Left: before adding the marked point, with visible faces emphasized.
Right: after the addition. Note that removing the marked point from
the convex hull on the right image, would result in the convex hull
that appears in the left image.}
\end{figure}

Formally, these planes are defined by some normal $n$ and a point
$p'$, and have $n\cdot\left(q-p'\right)\geq0$ for all $q\in DP\backslash p$
but $n\cdot\left(p-p'\right)\leq0$. Such a plane, when translated
so that its passes through $p$, is defined by $n\cdot\left(q-p\right)=0$.
We note that for any $q\in DP\backslash p$, we have 
\begin{eqnarray*}
n\cdot\left(q-p\right) & = & n\cdot\left(q-p+p'-p'\right)\\
 & = & n\cdot\left(q-p'\right)+\left(-n\cdot\left(p-p'\right)\right)\\
 & \geq & 0,
\end{eqnarray*}
which means the translated plane goes through $p$, supports $\mathcal{CH}\left(DP\right)$
and has a positive $n_{z}$ component (we assumed the original face
belonged to $\mathcal{LH}\left(DP\right)$). This means it satisfies
the requirements of Proposition \ref{prop:behind cant be a solution},
and therefore this plane (translated or not) cannot be a solution
to Problem \ref{fig:Central-Min-Max-problem.}, which concludes the
proof.
\end{proof}
With regard to the computational cost of these purgings, we note that
they are highly-parallelizable, in addition to being a single-pass
over the points.

\bibliographystyle{plainnat}
\bibliography{../Research/research}

\begin{thebibliography}{5}
\providecommand{\natexlab}[1]{#1}
\providecommand{\url}[1]{\texttt{#1}}
\expandafter\ifx\csname urlstyle\endcsname\relax
  \providecommand{\doi}[1]{doi: #1}\else
  \providecommand{\doi}{doi: \begingroup \urlstyle{rm}\Url}\fi

\bibitem[Bronstein et~al.(2006)Bronstein, Bronstein, and Kimmel]{GMDS}
Alexander~M. Bronstein, Michael~M. Bronstein, and Ron Kimmel.
\newblock Generalized multidimensional scaling: a framework for
  isometry-invariant partial surface matching.
\newblock \emph{Proc. National Academy of Sciences (PNAS)}, Volume
  103/5:\penalty0 1168--1172, January 2006.

\bibitem[de~Berg(2000)]{berg2000computational}
Mark de~Berg.
\newblock \emph{Computational geometry: algorithms and applications}.
\newblock Springer, 2000.
\newblock ISBN 9783540656203.
\newblock URL \url{http://books.google.com/books?id=C8zaAWuOIOcC}.

\bibitem[Fischer et~al.(2011)Fischer, G\"{a}rtner, Sch{\"o}nherr, and
  Wessendorp]{cgal:fgsw-lqps-11}
Kaspar Fischer, Bernd G\"{a}rtner, Sven Sch{\"o}nherr, and Frans Wessendorp.
\newblock Linear and quadratic programming solver.
\newblock In \emph{{CGAL} User and Reference Manual}. {CGAL Editorial Board},
  {3.8} edition, 2011.
\newblock
  http\://www.cgal.org/Manual/3.8/doc\_html/cgal\_manual/packages.html\#Pkg\:QPSolver.

\bibitem[Hough(1962)]{Hough}
Paul V.~C. Hough.
\newblock Method and means for recognizing complex patterns, March 1962.
\newblock URL \url{http://www.google.com/patents?vid=3069654}.

\bibitem[Megiddo(1984)]{Megiddo}
Nimrod Megiddo.
\newblock Linear programming in linear time when the dimension is fixed.
\newblock \emph{Journal of The ACM}, 31:\penalty0 114--127, 1984.
\newblock \doi{10.1145/2422.322418}.

\end{thebibliography}

\end{document}